\documentclass[12pt]{article}
\usepackage{amsmath,amsthm,amsfonts,amssymb}
\usepackage{booktabs}
\usepackage{algpseudocode}
\usepackage{varwidth}
\usepackage{colortbl}
\usepackage{enumerate}
\usepackage{hyperref}
\usepackage{pgfplotstable}
\usepackage{pgfplots}
\usepackage[a4paper,includeheadfoot,margin=2.54cm]{geometry}

\newtheorem{theorem}{Theorem}
\newtheorem{lemma}[theorem]{Lemma}

\theoremstyle{remark}

\theoremstyle{definition}

\newcommand{\FF}{\mathbb{F}}

\newcommand{\Sp}{Sp}

\title{Switching for Small Strongly Regular Graphs}
\author{Ferdinand Ihringer\thanks{Department of Mathematics: Analysis, Logic and Discrete Mathematics, Ghent University, Belgium, \href{mailto:ferdinand.ihringer@ugent.be}{ferdinand.ihringer@ugent.be}.
}}

\pgfplotstableset{int detect,
every even row/.style={before row={\rowcolor[gray]{0.9}}},
every head row/.style={before row=\toprule,after row=\midrule},
every last row/.style={after row=\bottomrule},
columns/nr/.style={column type=r,column name={SRGs}}}

\pgfplotsset{width=14cm,height=8cm,compat=1.16}
\pgfkeys{/pgf/number format/.cd,int detect}

\begin{document}


\maketitle

\begin{abstract}
We provide an abundance of strongly regular graphs (SRGs) for certain parameters $(n, k, \lambda, \mu)$ with $n < 100$. 
For this we use Godsil-McKay (GM) switching with a partition of type $4,n-4$ and Wang-Qiu-Hu (WQH) switching with a partition
of type $3,3,n-6$ or $4,4,n-8$. In most cases, we start with a highly symmetric graph which belongs to a finite geometry.
Many of the obtained graphs are new; for 
instance, we find \pgfmathprintnumber{16565438}
strongly regular graphs with parameters $(81, 30, 9, 12)$
while only 15 seem to be described in the literature.

We provide statistics about the size of the occurring automorphism groups.
We also find the recently discovered Kr\v{c}adinac partial geometry, thus finding a
third method of constructing it.
\end{abstract}

\section{Introduction}

\begin{quote}
  Strongly regular graphs lie on the cusp between highly structured and unstructured. For example, there is a unique strongly regular graph with parameters (36, 10, 4, 2), but there are 32548 non-isomorphic graphs with parameters (36, 15, 6, 6). \hfill {\footnotesize Peter Cameron, ``Random Strongly Regular Graphs?''\nocite{Cameron2003a}}
\end{quote}

A strongly regular graph (SRG) is a $k$-regular graph with $n$ vertices such that any two adjacent vertices have $\lambda$ common neighbors,
while any two non-adjacent vertices have $\mu$ common neighbors \cite{Hubaut1975}. 
The tuple $(n, k, \lambda, \mu)$ is called the parameter set of an SRG.
SRGs are interesting for many reasons. Their existence relates to several combinatorial
objects such as Steiner triple systems, quasi-symmetric designs, rank 3 permutation groups,
and partial geometries. See \cite{Brouwer2020} for a recent survey. They are also an important class of graphs for
isomorphism testing \cite{Babai1980,Spielman1996} as they are often hard to distinguish which makes it interesting to have many SRGs with the same parameters.

Our main aim is to provide an abundance of small SRGs which
can be used to test various conjectures in graph theory. For instance, researchers test conjectures by using
Spence's collection of SRGs \cite{SpenceHP}. Sometimes these are later refuted, cf.~\cite{Godsil2017}.
A larger selection of easily accessible SRGs as well as an 
easy method to generate them will hopefully lead to better conjectures.
An additional motivation is that 
\cite{Brouwer2020} cites some of the data of this document
and we want to provide a proper reference.

\medskip 

Let us recall special cases of Godsil-McKay (GM) switching \cite{Godsil1982} and Wang-Qiu-Hu (WQH) switching \cite{Wang2019}, cf.~\cite{Ihringer2019a}.

\begin{theorem}[GM Switching]\label{thm:gm}
    Let $\Gamma$ be a graph whose vertex set is partitioned as $C \cup D$.
    Assume that the induced subgraph on $C$ is regular. Suppose that each $x\in D$
    either has $0$, $|C|/2$, or $|C|$ neighbours in $C$.
    Construct a new graph $\overline{\Gamma}$ by switching adjacency and non-adjacency between $x \in D$
    and $C$ when $|\Gamma(x) \cap C| = |C|/2$.
    Then $\Gamma$ and $\overline{\Gamma}$ are cospectral.
\end{theorem}

\begin{theorem}[WQH Switching]\label{thm:wqh}
    Let $\Gamma$ be a graph whose vertex set is partitioned as $C_1 \cup C_2 \cup D$.
    Assume that the induced subgraphs on $C_1, C_2,$ and $C_1 \cup C_2$ are regular,
    and that the induced subgraphs on $C_1$ and $C_2$ have the same size and degree. Suppose that each $x\in D$
    either has the same number of neighbors in $C_1$ and $C_2$, or $\Gamma(x) \cap (C_1 \cup C_2) \in \{ C_1, C_2 \}$.
    Construct a new graph $\overline{\Gamma}$ by switching adjacency and non-adjacency between $x \in D$
    and $C_1 \cup C_2$ when $\Gamma(x) \cap (C_1 \cup C_2) \in \{ C_1, C_2 \}$.
    Then $\Gamma$ and $\overline{\Gamma}$ are cospectral.
\end{theorem}
Cospectral SRGs have the same parameters, so WQH switching applied to an SRG yields an SRG with the same parameters.
We say that we apply WQH switching with a partition of type $\ell,\ell,n-2\ell$ 
if $|C_1| = |C_2| = \ell$.
The aim of this paper is to provide a large collection of SRGs which can be generated by WQH
switching with a partition of type $2,2,n-4$, a partition of type $3,3,n-6$, or 
a partition of type $4,4,n-8$. 

Note that WQH switching with a partition of type
$2,2,n-4$ produces a graph isomorphic to a graph with Godsil-McKay switching 
if $C = C_1 \cup C_2$.
As this is mentioned in both \cite{Brouwer2020} 
and \cite{Ihringer2019a} without proof,
let us include one provided personally due to Munemasa \cite{Munemasa2018}.

\begin{lemma}\label{lem:am_proof}
  Theorem \ref{thm:gm} and Theorem \ref{thm:wqh}
  produce isomorphic graphs if $C = C_1 \cup C_2$ and $|C|=4$.
\end{lemma}
\begin{proof}[Proof due to Munemasa]
Let $I$ and $J$ denote the identity
matrix and all-ones matrix, respectively.
Let $P_1$ be the permutation matrix for $(1,2)(3,4)$,
$P_2$ the permutation matrix for $(1,3)(2,4)$,
and $P_3$ the permutation matrix for $(1,4)(2,3)$.
Put $Q_i = \frac12 (J - 2P_i)$.
Put
\begin{align*}
 &P = \begin{pmatrix} \frac12 (J -I) & 0\\ 0 & I\end{pmatrix},
 && R = \begin{pmatrix} Q_1 & 0\\ 0 & I\end{pmatrix}.
\end{align*}
Let $A$ be the adjacency matrix of $\Gamma$.
Suppose that $C = C_1 \cup C_2$
corresponds to the first four vertices of 
$A$. Then the graph $\overline{\Gamma}_1$ from 
Theorem \ref{thm:gm} has adjacency matrix $PAP$,
and the graph $\overline{\Gamma}_2$ from 
Theorem \ref{thm:wqh} has adjacency matrix $RAR$.
We have $Q_1Q_2Q_3 = \frac12 J-I$, and $Q_2Q_3 = P_2P_3$
is a permutation.
Hence, $PAP$ is a permutation of $RAR$.
Thus, $\overline{\Gamma}_1$ and $\overline{\Gamma}_2$
are isomorphic.
\end{proof}

It was shown 
in several papers, for instance \cite{Abiad2016,Ihringer2019a},
that GM and WQH switching work well for several families of SRGs. Here we present a more thorough investigation for small parameter sets.
Note that WQH switching was almost observed in Definition 3 of \cite{Behbahani2012} by Behbahani, Lam, and \"{O}sterg\aa{}rd.
This led to a similar investigation.

\medskip 

Table \ref{tab:total} summarizes our results.
Write GM($m$) (respectively, WQH$_\ell$($m$)) if we apply WQH switching up to $m$ times 
with a partition of type $2,2,n-4$ (respectively, $\ell,\ell,n-2\ell$) to our seed graph.

\begin{table}
\centering 
\pgfplotstabletypeset[row sep=\\,col sep=&,
.style={column type=r},
columns/Parameters/.style={string type,column type=l,column name={$(n,d,c,a)$}},
columns/Number/.style={column type=r,column name=\#},
columns/Type/.style={string type,column type=l},
columns/Seed/.style={string type,column type=l},
columns/Previously/.style={string type,column type=l,column name={$\gg$?}}
]{
Parameters & Number & Type & Seed & Previously \\
$(57,24,11,9)$ & 31490375 & GM($9$) & $S(2,3,19)$ & no\\
$(63, 30, 13, 15)$ & 13505292 & GM($5$) & $Sp(6, 2)$ & no \\
$(64,21,8,6)$ & 76323 & GM($\infty$) & $Bilin(2,3,2)$ & no \\
$(64,27,10,12)$ & 8613977 & GM($5$) & $VO^-(6,2)$ & yes \\
$(64,28,12,12)$ & 11063360 & GM($5$) & $VO^+(6,2)$ & maybe \\
$(70,27,12,9)$ & 78900835 & GM($10$) & $S(2,3,21)$ & no \\
$(81,24,9,6)$ & 7441608 & WQH$_3$($6$) & $VNO^+_4(3)$ & maybe \\
$(81,30,9,12)$ & 16565438 & WQH$_3$($\infty$) & $VNO^-_4(3)$ & yes\\
$(81,32,13,12)$ & 21392603 & WQH$_3$($6$) & $Bilin(2,2,3)$ & maybe\\
$(85,30,3,5)$ & 237787 & WQH$_4$($5$) & $Sp(4,4)$ & yes \\
$(96,19,2,4)$ & 178040 & WQH$_4$(6) & $Haemers(4)$ & maybe \\
$(96,20,4,4)$ & 133005 & WQH$_4$(6) & $GQ(5, 3)$ & maybe\\
}
\caption{The number of generated graphs.}
\label{tab:total}
\end{table}

Definitions of the graphs are in the corresponding subsections.
The last column ``$\gg$?'' contains a binary statement yes/no
to state whether (as far as the author is aware) 
the number of graphs 
constructed here is much larger than those
found in the literature. References
are given in the corresponding
subsections. We write ``maybe'' when
there are not many graphs in the literature,
but at least one construction, which in general 
is known to be prolific in some sense,
is associated with the given set of parameters.
Note that exact counts
are out of the scope of this note;
for instance, for parameters $(64, 27, 10, 12)$
there are at least 6 different methods
of constructing such SRGs, see \cite{Brouwer2020}, 
and it is not clear
how many nonisomorphic graphs these yield.

We provide the number of new graphs after each switching step and the automorphism group sizes for all graphs.
All graphs can be found on the homepage of the author in Nauty's graph6 format: \url{http://math.ihringer.org/srgs.php}.
There we also provide selected versions of the C program used.

\section{Finding Partitions and Other Technicalities}

Our investigation itself uses the folklore 
method of keeping a global record 
of {\it canonical representatives} of graphs
for isomorphy rejection, see \cite[\S4.2.1]{Kaski2006} for the general technique.

\smallskip
The canonical representative of a graph
is given by McKay's and Piperno's nauty-traces \cite{McKay2014}. 
A tiny self-written C program applies the switching.
We also use nauty-traces to calculate the sizes of the automorphism groups.
We use cliquer by Östergård \cite{OestergaardCliquer} to calculate clique numbers in some cases.
In two cases we use the default SRG with the corresponding parameters from Sage \cite{sagemath},
relying on Cohen's and Pasechnik's implementation of Brouwer's SRG database \cite{Cohen2017}.
Due to hardware constraints, we usually end our search at around 10 million SRGs.
A particular emphasis was put on parameters $(70,27,12,9)$ as the existence of a partial 
geometry $pg(6,6,4)$ is open.

\smallskip
We want to calculate all graphs which we can
obtain from a seed graph $\Gamma_0$
by applying a chosen type of switching 
up to $i$ times. We describe the general method
in the following:

\medskip 

\begin{algorithmic}[1]
\State 
    \begin{varwidth}[t]{0.7\textwidth}
Replace the seed graph $\Gamma_0$ 
by its canonical representative.
Note that there are many canonical forms for graphs
and one has to use the same method throughout
the whole algorithm.
\end{varwidth}
\medskip 

\State $T \gets \{ \Gamma_0 \}$, $C \gets \{ \Gamma_0 \}$, $j \gets 0$
\For{ $j < i$ }
  \State $N \gets \emptyset$
  \For{$\Gamma \in C$}
    \State 
    \begin{varwidth}[t]{0.7\textwidth}
    Calculate the set $M_\Gamma$ of graphs
    which can be obtained by applying the chosen switching
    to $\Gamma$. See \S\ref{sec:GM_alg}
    and \S\ref{sec:WQH_alg} for details.
    \end{varwidth}
    \medskip 
    \State 
    \begin{varwidth}[t]{0.7\textwidth}
    Calculate the set $N_\Gamma$ of canonical representatives
    of the graphs in $M_\Gamma$. 
    Note that $M_\Gamma$ might contain 
    distinct, but isomorphic graphs, while $N_\Gamma$ cannot.
    \end{varwidth}
    \medskip 
    \State $N \gets N \cup N_\Gamma$
  \EndFor
  \State $C \gets N \setminus T$
  \State $T \gets T \cup C$
  \State $j \gets j+1$
\EndFor
\State Now $T$ is the set of all canonical representatives
      of graphs which can be obtained from $\Gamma_0$
      by applying the chosen switching up to $i$ times.
\end{algorithmic}

\medskip 

%

Let us explain $T$, $C$, and $N$:
At the beginning of the outer for-loop,
$T$ (as in {\it total}) is the set of graphs after applying
the chosen switching $j$ times;
$C$ (as in {\it current}) is the set of graphs in $T$
to which the chosen switching was not yet applied.
The inner for-loop applies the chosen
switching to all elements in $C$ and collects
their canonical representatives
in $N$ (as in {\it new}).

\medskip 

Our partition finding method is very simple and described below. It uses simple pruning techniques.
Our vertex set is labeled $V=\{ 1, \ldots, n \}$ and the adjacency matrix of the graph is $A$.

\subsection{Type \texorpdfstring{$2,2,n-4$}{2,2,n-4}}\label{sec:GM_alg}

For WQH switching with a partition of type $2,2,n-4$, we implemented GM switching with a partition of type $4,n-4$.
The partition $C \cup D$ has to satisfy the following:
\begin{enumerate}[(A)]
 \item The induced subgraph on $C$ is regular.\label{it:gm1}
 \item All $x \in D$ satisfy $|\Gamma(x) \cap C| \in \{ 0, 2, 4 \}$.\label{it:gm2}
 \item There exists an $x\in D$ with $|\Gamma(x) \cap C| = 2$.\label{it:gm3}
\end{enumerate}
The last condition is not stated in Theorem \ref{thm:gm} above, 
but otherwise $\Gamma = \overline{\Gamma}$.

Most of our generated graphs have no symmetries,\footnote{
We have no a priori reason for this.} 
so we naively iterate through all $4$-tuples
$(c_1, c_2, c_3, c_4)$ with $c_1 < c_2 < c_3 < c_4$ in a nested loop.
We only check the conditions in the inner loop. 
First we check for \eqref{it:gm1} as it (naively) only involves accessing up to $|C|(|C|-1)=12$ entries of $A$,
while \eqref{it:gm2} and \eqref{it:gm3} might access up to $|D| \cdot |C| = 4(n-4)$ entries of $A$.

\subsection{Type \texorpdfstring{$\ell,\ell,n-2\ell$}{l,l,n-2l}}
\label{sec:WQH_alg}

For WQH switching with a partition of type $\ell,\ell,n-2\ell$, the partition $C_1 \cup C_2 \cup D$ has to satisfy the following:
\begin{enumerate}[(A)]
 \item The induced subgraph on $C_1$ is regular for some degree $k_1$.\label{it:wqh1}
 \item The induced subgraph on $C_2$ is regular with the same degree $k_1$.\label{it:wqh2}
 \item The bipartite subgraph between $C_1$ and $C_2$ (with the edges of $\Gamma$) is regular.\label{it:wqh3}
 \item All $x \in D$ satisfy $|\Gamma(x) \cap C_1| = |\Gamma(x) \cap C_2|$ or $\Gamma(x) \cap (C_1 \cup C_2) \in \{ C_1, C_2 \}$.\label{it:wqh4}
 \item The second case of \eqref{it:wqh4} occurs.\label{it:wqh5}
\end{enumerate}
While Theorem \ref{thm:wqh} asks for the induced subgraph in $C_1 \cup C_2$ to be regular, in light of \eqref{it:wqh1} and 
\eqref{it:wqh2}, testing for \eqref{it:wqh3} suffices and is faster.

\medskip 

Suppose that $C_1 = \{ c_1, \ldots, c_\ell \}$
and $C_2 = \{ c_{\ell+1}, \ldots, c_{2\ell} \}$.
We pick $c_1, \ldots, c_{2\ell}$ in order, where 
$c_1 < \ldots < c_\ell$ and 
$c_1 < c_{\ell+1} < \ldots < c_{2\ell}$.
Write $\tilde{C}_m = \{ c_1, \ldots, c_m \}$.

\smallskip 

Let $k_{11}(m)$ (respectively, $k_{22}(m)$) be the minimal degree 
of the induced subgraph on $\tilde{C}_m$
(respectively, $\tilde{C}_m \setminus \tilde{C}_\ell$)
and let $K_{11}(m)$ (respectively, $K_{22}(m)$) be the maximal degree
of the induced subgraph on $\tilde{C}_m$
(respectively, $\tilde{C}_m \setminus \tilde{C}_\ell$).
We discard $\tilde{C}_m$ if 
$K_1(m) - k_1(m) > \ell-m$ for $m \leq \ell$
as then \eqref{it:wqh1} is impossible.
Similarly, we discard $\tilde{C}_m$ if
$K_2(m) - k_2(m) > 2\ell-m$ for $m > \ell$
as then \eqref{it:wqh2} is impossible.

\smallskip 

Suppose $\{i,j\} = \{1,2\}$ and $m > \ell$.
Consider the bipartite graph with parts
$\tilde{C}_\ell$ and $\tilde{C}_m \setminus \tilde{C}_\ell$
(with the edges as in $\Gamma$).
Let $k_{12}(m)$ be the minimal degree
on $\tilde{C}_\ell$, $k_{21}(m)$
the minimal degree on $\tilde{C}_m \setminus \tilde{C}_\ell$,
$K_{12}(m)$ the maximum degree on $\tilde{C}_\ell$,
and $K_{12}(m)$ the maximum 
degree on $\tilde{C}_m \setminus \tilde{C}_\ell$.
We discard $\tilde{C}_m$ if $K_{21}(m) > k_{21}(m)$
(we already picked all vertices of $\tilde{C}_\ell$,
so all degrees in $\tilde{C}_m \setminus \tilde{C}_\ell$
must be the same by \eqref{it:wqh3}).
We also discard $\tilde{C}_m$ 
if $K_{12}(m) - k_{12}(m) > 2\ell - m$
as otherwise \eqref{it:wqh3} is impossible.

\smallskip 

For \eqref{it:wqh4} and \eqref{it:wqh5}
we only test in the inner loop after $C_1$ and $C_2$
are fully chosen.

\section{SRGs}\label{sec:data}

In this section we present the generated SRGs.
We apply switchings of type GM and WQH$_3$
to all the discussed graphs. If any of them 
does not work, then we try to apply 
switchings of type WQH$_4$.
We mention precisely the cases for which 
our technique produces SRGs
which are nonisomorphic to the used 
seed graph.

\subsection{Very Small Parameters}

SRGs with very small parameters are discussed in \cite{Behbahani2012}.
For instance, there are at least 342 SRGs with parameters $(49, 18, 7, 6)$ and one has a GM switching class of size 175.

\subsection{SRG(57,24,11,9)}

There is a one-to-one correspondence between Steiner triple systems and SRGs derived from a Steiner triple system \cite{Kaski2004,Spielman1996}.
Particularly, the complete classification of Steiner triple systems of order 19 \cite{Kaski2004} yields 
a large amount of SRGs with parameters $(57, 24, 11, 9)$. All of our graphs might be included in the 
$\pgfmathprintnumber{11084874829}$ SRGs of \cite{Kaski2004}.
Similarly to \cite[Theorem 4]{Ihringer2019a} one can see that certain cycle switches of designs (see \cite{Kaski2011})
can be interpreted as WQH switchings.
Particularly, the so-called Pasch switching corresponds to WQH switching with a partition
of type $2,2,n-4$, that is GM switching with a partition of type $4,n-4$.
To our knowledge, this is first observed in \cite{Behbahani2012}.
There it is also observed that GM switching can lead to non-geometric SRGs.
The authors of \cite{Behbahani2012} only 
find \pgfmathprintnumber{338536} SRGs by GM switching 
which is small compared to our number.

\smallskip

The number of generated SRGs after applying GM switching up to $i$ times can be found in Table \ref{tab:57_24_11_9}.

\begin{table}[htbp]
\centering
\pgfplotstabletypeset[font=\footnotesize,row sep=\\,col sep=&,
.style={column type=r},
columns/GM/.style={string type,column type=l,column name=$i$}
]{
GM    & 0 & 1 & 2  & 3  & 4  & 5  & 6 & 7 & 8 & 9\\
New   & 1 & 9 & 102 & 829 & 5408 & 31409 & 171607 & 913192 & 4826290 & 25541528 \\
Total & 1 & 10 & 112 & 941 & 6349 & 37758 & 209365 & 1122557 & 5948847 & 31490375 \\
}

\caption{SRGs with parameters $(57,24,11,9)$.}
\label{tab:57_24_11_9}
\end{table}

The automorphism group sizes 
can be found in Table \ref{tab:aut_57_24_11_9}.
The first row denotes the size of the automorphism
group, the second row denotes the numbers of SRGs
with an automorphism group of that size.

\pgfplotstableread{aut_szs_srg_57_24_11_9}\tabautSRGone
\pgfplotstabletranspose[header=true]\Target{\tabautSRGone}

\begin{table}[htbp]
\centering 
\pgfplotstabletypeset[
every head row/.style={ 
        output empty row,
        before row={%
            \toprule
        }
    }
,
columns/colnames/.style={string type,column type=l}
]{\Target}
\caption{Automorphism group sizes of SRGs with parameters $(57,24,11,9)$.}

\label{tab:aut_57_24_11_9}
\end{table}

\subsection{SRG(63, 30, 13, 15)}\label{sec:63_30_13_15}

Let us give a short description
of the collinearity graph of $\Sp(2d, q)$: Vertices are $1$-dimensional subspaces of $\FF_q^{2d}$. Two $1$-dimensional subspaces are adjacent
if they are perpendicular with respect to the bilinear form $x_1y_2 - x_2y_1 + \ldots + x_{2d-1}y_{2d} - x_{2d}y_{2d-1}$.
For $(d,q) = (3, 2)$, this graph has the desired parameters.
WQH switching works for $\Sp(2d, q)$, see \cite{Abiad2016} for $q=2$ and \cite{Ihringer2019a} for the general case.

The graph $\Sp(6, 2)$ has an automorphism group of size $\pgfmathprintnumber{1451520}$, clique number $7$ and coclique number $7$.
SRGs with the same parameters as $\Sp(6, 2)$ have spectrum $(30,3^{35},-5^{27})$, clique number at most $7$ and coclique number at most $9$.
More details on $\Sp(6, 2)$ can be found in \cite[\S10.21]{Brouwer2020}.

At most \pgfmathprintnumber{522079} SRGs are known from 
intersection-$8$ graphs of quasi-symmetric 
$2$-$(36, 16, 12)$ designs \cite{KV2016},
\pgfmathprintnumber{4653} SRGs with these parameters in \cite{Krcadinac2000}, 
at least 9 SRGs with these parameters in \cite{ABH2019},
and one more SRG in \cite{Abiad2016}. 

\smallskip 

The number of graphs after applying GM switching up to $i$ times
can be found in Table \ref{tab:63_30_13_15}.

\begin{table}[htbp]

\centering

\pgfplotstabletypeset[row sep=\\,col sep=&,
.style={column type=r},
columns/GM/.style={string type,column type=l,column name=$i$}
]{
GM    & 0 & 1 & 2  & 3    & 4      & 5\\
New   & 1 & 2 & 52 & 3275 & 254097 & 13247865 \\
Total & 1 & 3 & 55 & 3330 & 257427 & 13505292 \\
}
\caption{SRGs with parameters $(63, 30, 13, 15)$.}
\label{tab:63_30_13_15}
\end{table}

The automorphism group sizes can be found in Table \ref{tab:aut_63_30_13_15} and Figure \ref{fig:aut_63_30_13_15}.
The first column denotes the size of the automorphism
group, the second column denotes the numbers of SRGs
with an automorphism group of that size.

\pgfplotstableread{aut_szs_srg_63_30_13_15}\tabaut

\begin{table}[htbp]
\begin{center}
\footnotesize 
\pgfplotstabletypeset[
columns={sz,nr,sz,nr,sz,nr,sz,nr,sz,nr},
columns/sz/.style={column type=r,column name={$|G|$}},
display columns/0/.style={select equal part entry of={0}{5}},
display columns/1/.style={select equal part entry of={0}{5}},
display columns/2/.style={select equal part entry of={1}{5}},
display columns/3/.style={select equal part entry of={1}{5}},
display columns/4/.style={select equal part entry of={2}{5}},
display columns/5/.style={select equal part entry of={2}{5}},
display columns/6/.style={select equal part entry of={3}{5}},
display columns/7/.style={select equal part entry of={3}{5}},
display columns/8/.style={select equal part entry of={4}{5}},
display columns/9/.style={select equal part entry of={4}{5}}
]\tabaut
\end{center}
\caption{Automorphism group sizes of SRGs with parameters $(63, 30, 13, 15)$.}

\label{tab:aut_63_30_13_15}
\end{table}

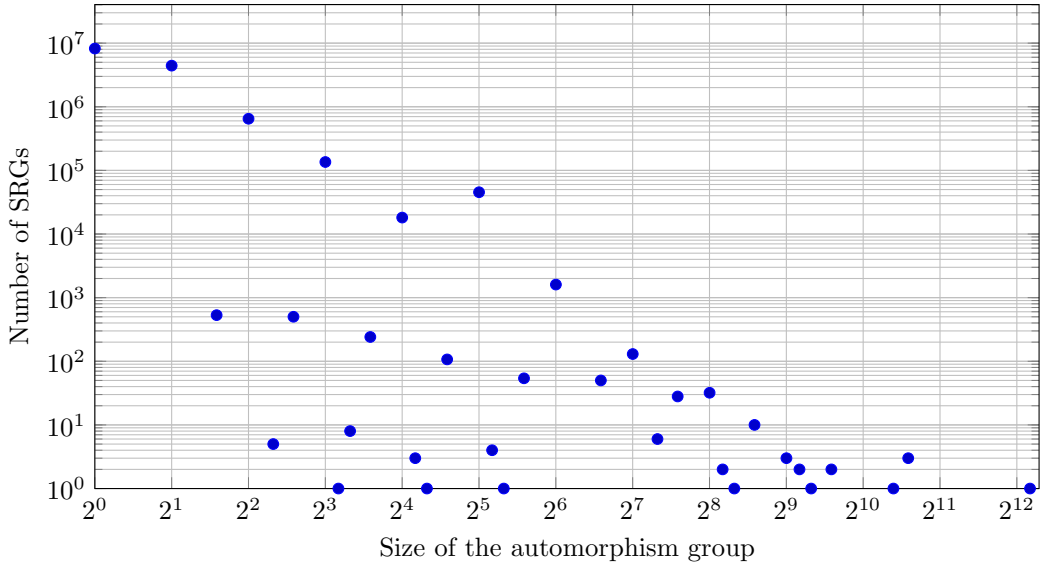
\begin{figure}[htbp]
\begin{center}
\footnotesize
\begin{tikzpicture}
\begin{loglogaxis}[
  xlabel=Size of the automorphism group,
  ylabel=Number of SRGs,
  grid=both,xmin=1,ymin=1,xmax=5000,log basis x=2]
\addplot table [only marks,y=nr, x=sz]{\tabaut};
\end{loglogaxis}
\end{tikzpicture}
\end{center}
\caption{Plotted automorphism group sizes.}
\label{fig:aut_63_30_13_15}
\end{figure}

The number of cliques of size $7$ can 
be found in Table \ref{tab:cl_63_30_13_15}. 
We also found graphs with no cliques 
of size $6$, but these examples are 
not reached in five steps.
While a partial geometry of type $pg(6,4,3)$ is known, none of the \pgfmathprintnumber{29017}
graphs with at least $45$ cliques belongs to a partial geometry.

\smallskip 

\pgfplotstableread{clq_szs_srg_63_30_13_15}\tabcl

\begin{table}[htbp]
\begin{center}
\footnotesize 
\pgfplotstabletypeset[
columns={sz,nr,sz,nr,sz,nr,sz,nr,sz,nr},
columns/sz/.style={column type=r,column name={Cls}},
display columns/0/.style={select equal part entry of={0}{6}},
display columns/1/.style={select equal part entry of={0}{6}},
display columns/2/.style={select equal part entry of={1}{6}},
display columns/3/.style={select equal part entry of={1}{6}},
display columns/4/.style={select equal part entry of={2}{6}},
display columns/5/.style={select equal part entry of={2}{6}},
display columns/6/.style={select equal part entry of={3}{6}},
display columns/7/.style={select equal part entry of={3}{6}},
display columns/8/.style={select equal part entry of={4}{6}},
display columns/9/.style={select equal part entry of={4}{6}},
display columns/10/.style={select equal part entry of={5}{6}},
display columns/11/.style={select equal part entry of={5}{6}}
]\tabcl
\end{center}
\caption{Cliques of size $7$ for parameters $(63, 30, 13, 15)$.}

\label{tab:cl_63_30_13_15}
\end{table}

\subsection{SRG(64,21,8,6)}

See Subsection \ref{sec:81_32_13_21} for a description of the graph $Bilin(2, 3, 2)$ which 
is our seed graph.
The search in \cite{Behbahani2012} found more than \pgfmathprintnumber{500000} SRGs, 
while the GM switching class of $Bilin(2, 3, 2)$ has only size \pgfmathprintnumber{76323}, therefore we omit any
further details.
We calculated the index chromatic number of a random subset of size 1000, but failed to find a counterexample
to the conjecture in \cite{Cioaba2018}, namely that the chromatic index of an SRG with $n$ even is always $k$ unless
the SRG is the Petersen graph.

\subsection{SRG(64,27,10,12)}

Let us give a short description of the graph $VO^-(2d,q)$.
Let $Q(x) = \alpha x_1^2 + \beta x_1x_2 + x_2^2 + \ldots + x_{2d}^2$ such that 
$\alpha x_1^2 + \beta x_1x_2 + x_2^2$ is irreducible over $\FF_q$. For $q=2$, we can choose
$(\alpha, \beta) = (1,1)$.
The vertices of $VO^-(2d,q)$ are the vectors of $\FF_q^{2d}$.
Two vertices $x,y$ are adjacent if $Q(x-y)=0$.
The graph $VO^-(6, 2)$ has an automorphism group of size \pgfmathprintnumber{3317760}.
More details on $VO^-(6, 2)$ can be found in \cite[\S10.25]{Brouwer2020}.

We find at least 9 SRGs with these parameters in \cite{ABH2019}.

\smallskip 

The number of graphs after applying GM switching up to $i$ times
can be found in Table \ref{tab:64_27_10_12}.

\begin{table}[htbp]
\begin{center}
\pgfplotstabletypeset[row sep=\\,col sep=&,
.style={column type=r},
columns/GM/.style={string type,column type=l,column name=$i$}
]{
GM    & 0 & 1 & 2  & 3  & 4  & 5  \\
New   & 1 & 2 & 43 & 2116 & 158036 & 8453779 \\
Total & 1 & 3 & 46 & 2162 & 160198 & 8613977 \\
}
\end{center}
\caption{SRGs with parameters $(64,27,10,12)$.}
\label{tab:64_27_10_12}
\end{table}

The automorphism group sizes are as in Table \ref{tab:aut_64_27_10_12}.
The first column denotes the size of the automorphism
group, the second column denotes the numbers of SRGs
with an automorphism group of that size.

\pgfplotstableread{aut_szs_srg_64_27_10_12}\tabautSRGfour
\begin{table}[htbp]
\begin{center}
\footnotesize
\pgfplotstabletypeset[
columns={sz,nr,sz,nr,sz,nr,sz,nr,sz,nr},
columns/sz/.style={column type=r,column name={$|G|$}},
display columns/0/.style={select equal part entry of={0}{5}},
display columns/1/.style={select equal part entry of={0}{5}},
display columns/2/.style={select equal part entry of={1}{5}},
display columns/3/.style={select equal part entry of={1}{5}},
display columns/4/.style={select equal part entry of={2}{5}},
display columns/5/.style={select equal part entry of={2}{5}},
display columns/6/.style={select equal part entry of={3}{5}},
display columns/7/.style={select equal part entry of={3}{5}},
display columns/8/.style={select equal part entry of={4}{5}},
display columns/9/.style={select equal part entry of={4}{5}}
]\tabautSRGfour
\end{center}
\caption{Automorphism group sizes of SRGs with parameters 
$(64, 27, 10, 12)$.}
\label{tab:aut_64_27_10_12}
\end{table}

The graph $VO^-(6, 2)$ is known to be $(K_5-e)$-free and its complement is $(K_7-e)$-free.
Therefore, it is a witness for the Ramsey number $R(K_5-e, K_7-e) \geq 65$, see \cite[\S10.25]{Brouwer2020}.
In fact, $R(K_5-e, K_7-e) = 65$.
Among the \pgfmathprintnumber{8613977} in our collection, it is the only graph with that property.
Indeed, it includes only 8 $K_5$-free graphs for which the complement is $K_7$-free.

\subsection{SRG(64,28,12,12)}

The graph $VO^+(2d,q)$ can be constructed the same way as $VO^-(2d,q)$ from the preceding section,
but with $(\alpha, \beta)$ chosen such that $\alpha x_1^2 + \beta x_1x_2 + x_2^2$ is reducible over $\FF_q$.
For $q=2$, we can choose $(\alpha, \beta)=(1,0)$.
The graph $VO^+(6, 2)$ has an automorphism group of size \pgfmathprintnumber{2580480}.
More details on $VO^+(6, 2)$ can be found in \cite[\S10.26]{Brouwer2020}.

We find a at least 9 SRGs with these parameters in \cite{ABH2019}.
We find 15 SRGs with these parameters in \cite{Kabanov2021}.

\smallskip 

The number of graphs after applying GM switching up to $i$ times
can be found in Table \ref{tab:64_28_12_12}.

\begin{table}[htbp]
\begin{center}
\pgfplotstabletypeset[row sep=\\,col sep=&,
.style={column type=r},
columns/GM/.style={string type,column type=l,column name=$i$}
]{
GM    & 0 & 1 & 2  & 3  & 4  & 5 \\
New   & 1 & 1 & 52 & 2680 & 201883 & 10858742 \\
Total & 1 & 3 & 55 & 2735 & 204618 & 11063360 \\
}
\end{center}
\caption{SRGs with parameters $(64,28,12,12)$.}
\label{tab:64_28_12_12}
\end{table}

The automorphism group sizes can be found in Table 
\ref{tab:aut_64_28_12_12}.
The first column denotes the size of the automorphism
group, the second column denotes the numbers of SRGs
with an automorphism group of that size.

\pgfplotstableread{aut_szs_srg_64_28_12_12}\tabautSRGfive
\begin{table}[htbp]
\begin{center}
\footnotesize 
\pgfplotstabletypeset[
columns={sz,nr,sz,nr,sz,nr,sz,nr},
columns/sz/.style={column type=r,column name={$|G|$}},
display columns/0/.style={select equal part entry of={0}{5}},
display columns/1/.style={select equal part entry of={0}{5}},
display columns/2/.style={select equal part entry of={1}{5}},
display columns/3/.style={select equal part entry of={1}{5}},
display columns/4/.style={select equal part entry of={2}{5}},
display columns/5/.style={select equal part entry of={2}{5}},
display columns/6/.style={select equal part entry of={3}{5}},
display columns/7/.style={select equal part entry of={3}{5}},
display columns/8/.style={select equal part entry of={4}{5}},
display columns/9/.style={select equal part entry of={4}{5}}
]\tabautSRGfive
\end{center}
\caption{Automorphism group sizes of SRGs with parameters 
$(64,28,12,12)$.}
\label{tab:aut_64_28_12_12}
\end{table}

\subsection{SRG(70,27,12,9)}

The classification of Steiner triple systems of order 21 \cite{Kokkala2020a,Kokkala2020} is a rich
source for a large number of SRGs with parameters $(70,27,12,9)$.
Maybe our list of \pgfmathprintnumber{78900835} SRGs is mostly disjoint 
to the \pgfmathprintnumber{13168639} SRGs in \cite{Kokkala2020} and the $\pgfmathprintnumber{83003869}$ SRGs in \cite{Kokkala2020a}
as the extra conditions in \cite{Kokkala2020a,Kokkala2020} appear to be restrictive.

Our seed graph belongs to a Steiner triple system on 21 points and has an automorphism group of size $126$.

\smallskip 

The number of graphs after applying GM switching up to $i$ times
can be found in Table \ref{tab:70_27_12_9}.

\begin{table}[htbp]
\begin{center}
\pgfplotstabletypeset[font=\footnotesize,row sep=\\,col sep=&,
.style={column type=r},
columns/GM/.style={string type,column type=l,column name=$i$}
]{
GM    & 0 & 1 & 2  & 3  & 4  & 5 & 6 & 7 & 8 & 9 & 10 \\
New   & 1 & 1 & 7 & 85 & 775 & 6094 & 43397 & 286285 & 1799283 & 10976064 & 65788843 \\
Total & 1 & 2 & 9 & 94 & 869 & 6963 & 50360 & 336645 & 2135928 & 13111992 & 78900835 \\
}
\end{center}
\caption{SRGs with parameters $(70, 27,12,9)$.}
\label{tab:70_27_12_9}
\end{table}

The automorphism group sizes can be found in Table \ref{tab:aut_70_27_12_9}.
The first row denotes the size of the automorphism
group, the second row denotes the numbers of SRGs
with an automorphism group of that size.

\pgfplotstableread{aut_szs_srg_70_27_12_9}\tabautSRGsix
\begin{table}[htbp]
\begin{center}
\pgfplotstabletypeset[
columns={sz,nr,sz,nr,sz,nr},
columns/sz/.style={column type=r,column name={$|G|$}},
display columns/0/.style={select equal part entry of={0}{3}},
display columns/1/.style={select equal part entry of={0}{3}},
display columns/2/.style={select equal part entry of={1}{3}},
display columns/3/.style={select equal part entry of={1}{3}},
display columns/4/.style={select equal part entry of={2}{3}},
display columns/5/.style={select equal part entry of={2}{3}}
]\tabautSRGsix
\end{center}
\caption{Automorphism group sizes of SRGs with parameters $(70,27,12,9)$.}
\label{tab:aut_70_27_12_9}
\end{table}

The complement of an SRG with parameters $(70, 27, 12, 9)$
can be the point graph of a partial geometry of type $pg(6,6,4)$.
An SRG belonging to such a partial geometry
has at least $70$ 
cocliques of size $7$ which pairwise meet in at most one vertex.
In the following, we list the number of the cocliques of size $7$.

\begin{table}[htbp]
\pgfplotstableread{coclq_szs_srg_70_27_12_9}\tabcocl
\begin{center}
\footnotesize 
\pgfplotstabletypeset[
columns={sz,nr,sz,nr,sz,nr,sz,nr,sz,nr},
columns/sz/.style={column type=r,column name={CoCls}},
display columns/0/.style={select equal part entry of={0}{5}},
display columns/1/.style={select equal part entry of={0}{5}},
display columns/2/.style={select equal part entry of={1}{5}},
display columns/3/.style={select equal part entry of={1}{5}},
display columns/4/.style={select equal part entry of={2}{5}},
display columns/5/.style={select equal part entry of={2}{5}},
display columns/6/.style={select equal part entry of={3}{5}},
display columns/7/.style={select equal part entry of={3}{5}},
display columns/8/.style={select equal part entry of={4}{5}},
display columns/9/.style={select equal part entry of={4}{5}}
]\tabcocl
\end{center}
\caption{Cocliques of size $7$ for parameters $(70,27,12,9)$.}
\label{tab:cl_70_27_12_9}
\end{table}

We found only two graphs with a sufficient amount of cocliques.
One has an automorphism group of size 6, one of size 1.
Both have at most 16 cocliques which pairwise meet in at most
one vertex. Hence, we do not obtain a $pg(6,6,4)$.

\subsection{SRG(81,24,9,6)}

A nice geometric graph with the given parameters can be obtained as follows.
Let $Q(x) = x_1^2-x_2^2+x_3^2+x_4^2$. The vertices are the vectors of $\FF_3^4$.
Two vertices $x$ and $y$ are adjacent if $Q(x-y) = 1$. This graph is also
known as $VNO^+(4,3)$ and has an automorphism group of size \pgfmathprintnumber{93312}.

We find 13 graphs with these parameters in \cite{Behbahani2012}.

\smallskip 

The number of graphs after applying WQH switching with a partition 
of type $3,3,n-6$ up to $i$ times can be found in Table 
\ref{tab:81_24_9_6}.

\begin{table}[htbp]
\begin{center}
\pgfplotstabletypeset[row sep=\\,col sep=&,
.style={column type=r},
columns/GM/.style={string type,column type=l,column name=$i$}
]{
GM    & 0 & 1 & 2  & 3  & 4 & 5 & 6  \\
New   & 1 & 2 & 31 & 596 & 15183 & 377270 & 7048525 \\
Total & 1 & 3 & 34 & 630 & 15813 & 393083 & 7441608 \\
}
\end{center}
\caption{SRGs with parameters $(81,24,9,6)$.}
\label{tab:81_24_9_6}
\end{table}

The automorphism group sizes can be found in Table 
\ref{tab:aut_81_24_9_6}.
The first column denotes the size of the automorphism
group, the second column denotes the numbers of SRGs
with an automorphism group of that size.

\pgfplotstableread{aut_szs_srg_81_24_9_6}\tabautSRGseven
\begin{table}[htbp]
\begin{center}
\footnotesize 
\pgfplotstabletypeset[
columns={sz,nr,sz,nr,sz,nr,sz,nr,sz,nr},
columns/sz/.style={column type=r,column name={$|G|$}},
display columns/0/.style={select equal part entry of={0}{5}},
display columns/1/.style={select equal part entry of={0}{5}},
display columns/2/.style={select equal part entry of={1}{5}},
display columns/3/.style={select equal part entry of={1}{5}},
display columns/4/.style={select equal part entry of={2}{5}},
display columns/5/.style={select equal part entry of={2}{5}},
display columns/6/.style={select equal part entry of={3}{5}},
display columns/7/.style={select equal part entry of={3}{5}},
display columns/8/.style={select equal part entry of={4}{5}},
display columns/9/.style={select equal part entry of={4}{5}}
]\tabautSRGseven
\end{center}
\caption{Automorphism group sizes of SRGs with parameters
$(81,24,9,6)$.}
\label{tab:aut_81_24_9_6}
\end{table}

\subsection{SRG(81,30,9,12)}

Van Lint and Schrijver discovered a partial geometry of type $pg(5,5,2)$
\cite{Lint1981}, the vL-S partial geometry. The point graph of this partial geometry is an SRG with 
parameters $(81, 30, 9, 12)$. Recently, a second partial geometry of the same
type was discovered by Kr\v{c}adinac \cite{Krcadinac2020} and,
almost at the same time, by Crnkovi\'{c}, \v{S}vob and Tonchev \cite{Crnkovic2020}.
More details on $VNO^-(4,3)$ can be found in \cite[\S10.29]{Brouwer2020}.

We can describe the SRG derived from the vL-S geometry as follows.
Let $Q(x) = x_1^2+x_2^2+x_3^2+x_4^2$. The vertices are the vectors of $\FF_3^4$.
Two vertices $x$ and $y$ are adjacent if $Q(x-y) = 1$. This graph is also
known as $VNO^-_4(3)$.

The number of graphs after applying WQH switching 
with a partition of type $3,3,n-6$ up to $i$ times
can be found in Table \ref{tab:81_30_9_12}.
Further applications of the switching operation do not yield 
more graphs.

\begin{table}[htbp]
\begin{center}
\pgfplotstabletypeset[row sep=\\,col sep=&,
.style={column type=r},
columns/GM/.style={string type,column type=l,column name=$i$}
]{
GM    & 0 & 1 & 2  & 3   & 4    & 5     & 6      & 7      \\
New   & 1 & 2 & 21 & 144 & 1249 & 12560 & 107665 & 691650 \\
Total & 1 & 3 & 24 & 168 & 1417 & 13977 & 121642 & 813292 \\
}
\end{center}
\begin{center}
\pgfplotstabletypeset[row sep=\\,col sep=&,
.style={column type=r},
columns/GM/.style={string type,column type=l,column name=$i$}
]{
GM    & 8       & 9        & 10       & 11       & 12 \\
New   & 2957467 & 7041075  & 4892852  & 835010   & 25742 \\
Total & 3770759 & 10811834 & 15704686 & 16539696 & 16565438 \\
}
\end{center}

\caption{SRGs with parameters $(81,30,9,12)$.}
\label{tab:81_30_9_12}
\end{table}

The automorphism group sizes can be found 
in Table \ref{tab:aut_81_30_9_12}.
The first column denotes the size of the automorphism
group, the second column denotes the numbers of SRGs
with an automorphism group of that size.

\pgfplotstableread{aut_szs_srg_81_30_9_12}\tabautSRGeight
\begin{table}[htbp]
\begin{center}
\footnotesize 
\pgfplotstabletypeset[
columns={sz,nr,sz,nr,sz,nr,sz,nr},
columns/sz/.style={column type=r,column name={$|G|$}},
display columns/0/.style={select equal part entry of={0}{5}},
display columns/1/.style={select equal part entry of={0}{5}},
display columns/2/.style={select equal part entry of={1}{5}},
display columns/3/.style={select equal part entry of={1}{5}},
display columns/4/.style={select equal part entry of={2}{5}},
display columns/5/.style={select equal part entry of={2}{5}},
display columns/6/.style={select equal part entry of={3}{5}},
display columns/7/.style={select equal part entry of={3}{5}},
display columns/8/.style={select equal part entry of={4}{5}},
display columns/9/.style={select equal part entry of={4}{5}}
]\tabautSRGeight
\end{center}
\caption{Automorphism group sizes of SRGs with parameters
$(81,30,9,12)$.}
\label{tab:aut_81_30_9_12}
\end{table}

Our seed graph, the point graph of the vL-S partial geometry
has an automorphism group of size \pgfmathprintnumber{116640}. By comparing
automorphism group sizes, we see that our list cannot contain all of the 14
new SRGs described in \cite{Crnkovic2020}.

A partial geometry $pg(5,5,2)$ necessarily has at least 81 cliques of size 6 
which pairwise meet in at most one vertex.
Our search produced 38 SRGs with sufficiently many cliques of size 6, see Table \ref{tab:cl_81_30_9_12}. Only the ones
corresponding to the vL-S partial geometry and the Kr\v{c}adinac partial geometry
are point graphs of partial geometries. Hence, we rediscover the
Kr\v{c}adinac partial geometry via a third method.
The distance between the vL-S partial geometry and the Kr\v{c}adinac partial geometry
is 6 using WQH switchings with $|C_1|=|C_2|=3$.

\pgfplotstableread{clq_szs_srg_81_30_9_12}\tabclqsSRGeight
\begin{table}[htbp]
\begin{center}
\footnotesize 
\pgfplotstabletypeset[
columns={sz,nr,sz,nr,sz,nr,sz,nr,sz,nr,sz,nr},
columns/sz/.style={column type=r,column name={Cls}},
display columns/0/.style={select equal part entry of={0}{6}},
display columns/1/.style={select equal part entry of={0}{6}},
display columns/2/.style={select equal part entry of={1}{6}},
display columns/3/.style={select equal part entry of={1}{6}},
display columns/4/.style={select equal part entry of={2}{6}},
display columns/5/.style={select equal part entry of={2}{6}},
display columns/6/.style={select equal part entry of={3}{6}},
display columns/7/.style={select equal part entry of={3}{6}},
display columns/8/.style={select equal part entry of={4}{6}},
display columns/9/.style={select equal part entry of={4}{6}},
display columns/10/.style={select equal part entry of={5}{6}},
display columns/11/.style={select equal part entry of={5}{6}}
]\tabclqsSRGeight
\end{center}
\caption{Cliques of size $6$ for parameters $(81,30,9,12)$.}
\label{tab:cl_81_30_9_12}
\end{table}


\subsection{SRG(81, 32, 13,12)}\label{sec:81_32_13_21}

The graph $Bilin(2, m-2, q)$, $n \geq 4$, can be described as follows.
The vertices are the set of all $2$-spaces of $\FF_q^{m}$ which are disjoint
to a fixed $(m-2)$-space. Two $2$-spaces are adjacent if their meet is a $1$-space.
This yields an SRG. For $n=4$ and $q=3$, its parameters are $(81, 32, 13, 12)$
and it has an automorphism group of size \pgfmathprintnumber{186624}.

\smallskip 

The number of graphs after applying WQH switching with a partition 
of type $3,3,n-6$ up to $i$ times can be found in Table
\ref{tab:81_32_13_12}.
The automorphism group sizes can be found in Table 
\ref{tab:aut_81_32_13_12}.
The first row denotes the size of the automorphism
group, the second row denotes the numbers of SRGs
with an automorphism group of that size.

\begin{table}[htbp]
\begin{center}
\pgfplotstabletypeset[row sep=\\,col sep=&,
.style={column type=r},
columns/GM/.style={string type,column type=l,column name=$i$}
]{
GM    & 0 & 1 & 2  & 3   & 4     & 5      & 6      \\
New   & 1 & 2 & 41 & 963 & 29120 & 841699 & 20520777\\
Total & 1 & 3 & 44 & 1007 & 30127 & 871826 & 21392603 \\
}
\end{center}
\caption{SRGs with parameters $(81,32,13,12)$.}
\label{tab:81_32_13_12}
\end{table}

\pgfplotstableread{aut_szs_srg_81_32_13_12}\tabautSRGnine
\begin{table}[htbp]
\begin{center}
\footnotesize
\pgfplotstabletypeset[
columns={sz,nr,sz,nr,sz,nr,sz,nr,sz,nr},
columns/sz/.style={column type=r,column name={$|G|$}},
display columns/0/.style={select equal part entry of={0}{5}},
display columns/1/.style={select equal part entry of={0}{5}},
display columns/2/.style={select equal part entry of={1}{5}},
display columns/3/.style={select equal part entry of={1}{5}},
display columns/4/.style={select equal part entry of={2}{5}},
display columns/5/.style={select equal part entry of={2}{5}},
display columns/6/.style={select equal part entry of={3}{5}},
display columns/7/.style={select equal part entry of={3}{5}},
display columns/8/.style={select equal part entry of={4}{5}},
display columns/9/.style={select equal part entry of={4}{5}}
]\tabautSRGnine
\end{center}
\caption{Automorphism group sizes of SRGs with parameters 
$(81, 32, 13, 12)$.}
\label{tab:aut_81_32_13_12}
\end{table}

\subsection{SRG(85, 20, 3, 5)}\label{sec:85_20_3_5}

See Subsection \ref{sec:63_30_13_15} for a description of the graph $Sp(4, 4)$.
It has an automorphism group of size \pgfmathprintnumber{1958400}.

\smallskip 

The number of graphs after applying WQH switching with 
a partition of type $4,4,n-8$ up to $i$ times
can be found in Table \ref{tab:85_20_3_5}.
Van Dam and Guo provide \pgfmathprintnumber{127433}
graphs with parameters $(85, 20, 3, 5)$
in \cite{vDG2022}. The list of graphs here shares precisely
\pgfmathprintnumber{3501} entries with their list.

\begin{table}[htbp]
\begin{center}
\pgfplotstabletypeset[row sep=\\,col sep=&,
.style={column type=r},
columns/GM/.style={string type,column type=l,column name=$i$}
]{
GM    & 0 & 1 & 2  & 3   & 4     & 5      \\
New   & 1 & 1 & 16 & 442 & 12303 & 225024 \\
Total & 1 & 2 & 18 & 460 & 12763 & 237787 \\
}
\end{center}
\caption{SRGs with parameters $(85,20,3,5)$.}
\label{tab:85_20_3_5}
\end{table}

The automorphism group sizes can be found 
in Table \ref{tab:aut_85_20_3_5}.
The first column denotes the size of the automorphism
group, the second column denotes the numbers of SRGs
with an automorphism group of that size.

\pgfplotstableread{aut_szs_srg_85_20_3_5}\tabautSRGten
\begin{table}[htbp]
\begin{center}
\footnotesize 
\pgfplotstabletypeset[
columns={sz,nr,sz,nr,sz,nr,sz,nr},
columns/sz/.style={column type=r,column name={$|G|$}},
display columns/0/.style={select equal part entry of={0}{5}},
display columns/1/.style={select equal part entry of={0}{5}},
display columns/2/.style={select equal part entry of={1}{5}},
display columns/3/.style={select equal part entry of={1}{5}},
display columns/4/.style={select equal part entry of={2}{5}},
display columns/5/.style={select equal part entry of={2}{5}},
display columns/6/.style={select equal part entry of={3}{5}},
display columns/7/.style={select equal part entry of={3}{5}},
display columns/8/.style={select equal part entry of={4}{5}},
display columns/9/.style={select equal part entry of={4}{5}}
]\tabautSRGten
\end{center}
\caption{Automorphism group sizes of SRGs with 
parameters $(85,20,3,5)$.}
\label{tab:aut_85_20_3_5}
\end{table}

\subsection{SRG(96, 19, 2, 4)}\label{sec:96_19_2_4}

See \cite[\S8.A]{Brouwer1982} for a construction of graphs of
type Haemers($q$). Note that even for fixed $q$, this does not
uniquely determine the graph. Our seed graph has an automorphism
group of size \pgfmathprintnumber{9216}.

In \cite{GMV2006} we find 2 graphs with these parameters.
Surely, there are many more as several constructions
are known and the constructions of type Haemers($4$) 
allow for some freedom.

\smallskip 

The number of graphs after applying WQH switching with 
a partition of type $4,4,n-8$ up to $i$ times
can be found in Table \ref{tab:96_19_2_4}.

\begin{table}[htbp]
\begin{center}
\pgfplotstabletypeset[row sep=\\,col sep=&,
.style={column type=r},
columns/GM/.style={string type,column type=l,column name=$i$}
]{
GM    & 0 & 1 & 2  & 3   & 4     & 5     & 6   \\
New   & 1 & 2 & 17 & 160 & 1680  & 17578 & 158602 \\
Total & 1 & 3 & 20 & 180 & 1860  & 19438 & 178040 \\
}
\end{center}
\caption{SRGs with parameters $(96,19,2,4)$.}
\label{tab:96_19_2_4}
\end{table}

The automorphism group sizes can be found in Table 
\ref{tab:aut_96_19_2_4}.
The first column denotes the size of the automorphism
group, the second column denotes the numbers of SRGs
with an automorphism group of that size.

\pgfplotstableread{aut_szs_srg_96_19_2_4}\tabautSRGeleven
\begin{table}[htbp]
\begin{center}
\footnotesize 
\pgfplotstabletypeset[
columns={sz,nr,sz,nr,sz,nr,sz,nr},
columns/sz/.style={column type=r,column name={$|G|$}},
display columns/0/.style={select equal part entry of={0}{5}},
display columns/1/.style={select equal part entry of={0}{5}},
display columns/2/.style={select equal part entry of={1}{5}},
display columns/3/.style={select equal part entry of={1}{5}},
display columns/4/.style={select equal part entry of={2}{5}},
display columns/5/.style={select equal part entry of={2}{5}},
display columns/6/.style={select equal part entry of={3}{5}},
display columns/7/.style={select equal part entry of={3}{5}},
display columns/8/.style={select equal part entry of={4}{5}},
display columns/9/.style={select equal part entry of={4}{5}}
]\tabautSRGeleven
\end{center}
\caption{Automorphism group sizes of SRGs with 
parameters $(96,19,2,4)$.}
\label{tab:aut_96_19_2_4}
\end{table}

\subsection{SRG(96, 20, 4, 4)}\label{sec:96_20_4_4}

Our seed graph is the point graph of the 
unique generalized quadrangle of order $(5, 3)$
and has a group of size \pgfmathprintnumber{138240}.
In \cite{GMV2006} we find 6 graphs with these parameters.
Surely, there are many more as plenty constructions
are known, but we are unaware of any counts.

The number of graphs after applying WQH switching with 
a partition of type $4,4,n-8$ up to $i$ times
can be found in Table \ref{tab:96_20_4_4}.

\begin{table}[htbp]
\begin{center}
\pgfplotstabletypeset[row sep=\\,col sep=&,
.style={column type=r},
columns/GM/.style={string type,column type=l,column name=$i$}
]{
GM    & 0 & 1 & 2  & 3   & 4     & 5     & 6   \\
New   & 1 & 2 & 13 & 95  & 949   & 10773 & 121172 \\
Total & 1 & 3 & 16 & 111 & 1060  & 11833 & 133005 \\
}
\end{center}
\caption{SRGs with parameters $(96, 20, 4, 4)$.}
\label{tab:96_20_4_4}
\end{table}

The automorphism group sizes can be found 
in Table \ref{tab:aut_96_20_4_4}.
The first column denotes the size of the automorphism
group, the second column denotes the numbers of SRGs
with an automorphism group of that size.

\pgfplotstableread{aut_szs_srg_96_20_4_4}\tabautSRGtwelve
\begin{table}[htbp]
\begin{center}
\footnotesize
\pgfplotstabletypeset[
columns={sz,nr,sz,nr,sz,nr,sz,nr},
columns/sz/.style={column type=r,column name={$|G|$}},
display columns/0/.style={select equal part entry of={0}{5}},
display columns/1/.style={select equal part entry of={0}{5}},
display columns/2/.style={select equal part entry of={1}{5}},
display columns/3/.style={select equal part entry of={1}{5}},
display columns/4/.style={select equal part entry of={2}{5}},
display columns/5/.style={select equal part entry of={2}{5}},
display columns/6/.style={select equal part entry of={3}{5}},
display columns/7/.style={select equal part entry of={3}{5}},
display columns/8/.style={select equal part entry of={4}{5}},
display columns/9/.style={select equal part entry of={4}{5}}
]\tabautSRGtwelve
\end{center}
\caption{Automorphism group sizes of SRGs 
with parameters $(96, 20, 4, 4)$.}
\label{tab:aut_96_20_4_4}
\end{table}

\section{Future Work}

It might be very fruitful to use switching to optimize SRGs for a certain parameter. For instance, switching embeds 
naturally in a threshold accepting algorithm. Note that for Steiner triple systems, one can find a similar 
suggestion in \cite{Kaski2011}.

Our investigation is incomplete in at least two ways. Firstly, we might not have checked all known SRGs with less than 100
vertices for the considered switchings
(as there are too many constructions known).
Secondly, surely there exist SRGs for some of the parameters which are at the time of writing unknown.
Here is a list of all sets of parameters for which SRGs are known, but we failed at finding a graph for
which our switching works. Note that we did not investigate parameter sets which are completely classified.
\begin{align*}
 &(37, 18, 8, 9), &&(41, 20, 9, 10), && (45, 22, 10, 11),&& (49, 24, 11, 12),\\
 &(50, 21, 8, 9), &&(53, 26, 12, 13), &&(65, 32, 15, 16),&& (70, 27, 12, 9), \\
 &(73, 36, 17, 18),&&(81, 40, 19, 20), &&(82, 36, 15, 16),&& (89, 44, 21, 22)\\
 &(97, 48, 23, 24), &&(99, 48, 22, 24).
\end{align*}

\paragraph*{Acknowledgements} The author is supported by a postdoctoral fellowship of the Research Foundation - Flanders (FWO).
The author thanks Andries E.~Brouwer, Gordon Royle, and the second referee for comments and suggestions on earlier drafts of this paper.
The author thanks Akihiro Munemasa for the proof of Lemma \ref{lem:am_proof}.

\appendix

\section{Ambiguous Seed Graphs}

Here we list the used seed graphs for the less beautiful seed graphs.

\subsection{SRG(57,24,11,9)}

{\footnotesize
\begin{verbatim*}
x`MjkWRlZLZHuJY^J]~NvkT?^_KB[Sl_LAimY_Wxy_WFSGj`M_zopIn|?ZeaYYMo{Ceu
NC\Lap{\_]???^~{LiEiOKMaiaLQISj?taGIsd[cIMQLRoHMSMq[`wg@@PUDl@xpcG[p
@|QlDCSeedCiHOrJ_yOOwdzLASGw`zhrE_OjhCwlACKySW?Q@|[ouOBOBtuPrHaFGHuO
eb?sYB[ob`KOe_u@rbE@jGHoMY[@p{_deCTgBbe_qqCVgBb_]E@rK{?????~~~~
\end{verbatim*}
\par}

\subsection{SRG(70,27,12,9)}

{\footnotesize
\begin{verbatim}
~?@EQd_pJPwdUi{chWU`w]W^hm`Xt?}pX^HYwRu\G}WF~~sLX?Sdp?kd_HmLGAtS`BFT
WGdYLGqAy[HaIpe]?NADaiCLQcXOc}pO_`Z{aaCI^eGbhlGXCRTqcF_@}?N{RWiY[QPf
Eal`_???F~~{TDQGdcAIbKCXaAdOa`[cBUgGSMe?ss[_EIOakhi?cW`SqET?`ESzpHe@
?havWEX_oEWWXURAHe_ChDpg_rW?YNoo?pwK?oxkKQeABD[PqQLHADAyakGXOCHoI}V@
Dg?`i@^\BgJQ_GbqIIV@RD@@UooE?@hQ_|h_tH?ASl?zUOy?R@PhOJ_FeoBGITgCw@xr
?eAa?^JUORW@cDG@{ueAMW?F`tgC\KCMg?F`ySAiqAVS?Bow@{c?{o~???????~~~~w
\end{verbatim}
\par}

\subsection{SRG(96,19,2,4)}

{\footnotesize
\begin{verbatim}
~?@_????CA?[_K?C?G?S_Ao???q@cACo?W?AA??OO??oEHC_WQCCGOGGCO_?Ew?CQ?d?P
_G[?S?d?@O?s?ScY??D`J??@EGE@OPG`GDE?@o??@OD@_`A_GaP@PC??a_aG_?OSH`IA?
Q@@GSGAOC?WBD?S?GWCg_g?Q?CgOPJ?OAEGAKoCCSAH?OgGGgHCAD?cO_CPG@_W_OG_oA
OAGDCO@C?`C?oa?__Pg@?O?oCC`OG@?C__`?`AG`?AC?x?cD@?CA@WB?wKe?K?Ka@__aK
g?SaGQGAaQ_DAEE?M??_^_??GOAQC`I@O?C_@cAQDA_??EECKe?W?WKKA_SI?@G?AQOWD
@_?E?B??N~??????F}???????@}@QpO?OOa?K?GUcW?OOa?H?CoCS?@OaaG_CE@D??DAI
Ga?OAa_PC@@CY?@?II@@GC@Cw?C?C`X?`GCU????CHQOGWCS_???G?????????F~~?BBo
@_?oK?w_AGPGPa?_O@?QDAHCPDG?`?C?SD@_??@~w????Ko?I????AYeT@CB@@O????XU
Y_`@GO@~_?????@xw?@oOIOOB_Gp?QG?OOOh?OF?aP?QG@?IR?E_GQA?OGK?_Ap_CgGO_
_OCH?_DG_hO`?G?gA_@??UGGdAOA?I?a?O?EPPPCG@O`???iG?BDDAO_D?c??AgO????B
_?wW?Fo?B_?K??AHO_P_gs?QOA_??@@aOGcDJ?BC@O???O@KE?]?wB?_?o???????????
?@~~~
\end{verbatim}
\par}


\end{document}